\documentclass[11pt]{article}

\usepackage[margin=1in]{geometry}

\usepackage[colorlinks=true,linkcolor=blue,citecolor=blue,urlcolor=blue]{hyperref}

\usepackage{enumitem}
\usepackage{amsmath,amssymb,amsthm,mathtools}

\usepackage[nameinlink,capitalize]{cleveref}
\usepackage{algorithm}
\usepackage[noend]{algpseudocode}
\usepackage{float}
\usepackage{placeins}
\usepackage{xcolor}

\usepackage{aliascnt}
\usepackage[nameinlink,capitalize]{cleveref}
\newtheorem{theorem}{Theorem}
\newaliascnt{lemma}{theorem}
\newtheorem{lemma}[lemma]{Lemma}
\aliascntresetthe{lemma}
\newaliascnt{corollary}{theorem}
\newtheorem{corollary}[corollary]{Corollary}
\aliascntresetthe{corollary}
\newaliascnt{remark}{theorem}
\newtheorem{remark}[remark]{Remark}
\aliascntresetthe{remark}
\newaliascnt{definition}{theorem}
\newtheorem{definition}[definition]{Definition}
\aliascntresetthe{definition}
\newaliascnt{assumption}{theorem}
\newtheorem{assumption}[assumption]{Assumption}
\aliascntresetthe{assumption}
\newcommand{\E}{\mathbb E}
\newcommand{\Prb}{\mathbb P}

\newcommand{\calB}{\mathcal B}
\newcommand{\calP}{\mathcal P}

\newcommand{\cA}{\mathcal A}
\newcommand{\polylog}{\operatorname{polylog}}
\newcommand{\vbl}{\operatorname{vbl}}

\title{Streaming Hypergraph Coloring via Palette Sparsification}
\author{
  Artur Czumaj\thanks{Department of Computer Science and DIMAP, University of Warwick. Email: \texttt{A.Czumaj@warwick.ac.uk}.}
  \and
  Pan Peng\thanks{School of Computer Science and Technology, University of Science and Technology of China. Email: \texttt{ppeng@ustc.edu.cn}.}
  \and
  Ruizhe Shi\thanks{Paul G. Allen School of Computer Science \& Engineering, University of Washington. Email: \texttt{zhezi@cs.washington.edu}.}
  \and
  Christian Sohler\thanks{Department of Mathematics and Computer Science, University of Cologne. Email: \texttt{csohler@uni-koeln.de}.}
}
\date{}

\begin{document}
\hypersetup{pageanchor=false}
\maketitle

\begin{abstract}
For every fixed $k\ge2$, we give a randomized one-pass insertion-only algorithm that colors an $n$-vertex $k$-uniform hypergraph of maximum degree $\Delta$ with $O(\Delta^{1/(k-1)})$ colors using $\widetilde O_k(n)$ bits of working memory. 
As a graph-theoretic result of independent interest, we also prove a tight palette-sparsification theorem for general uniform hypergraphs.
Independently sampled lists of $\Theta(\sqrt{\log n})$ colors from a palette of size $O(\Delta^{1/(k-1)})$ preserve colorability with high probability; the list-size dependence is asymptotically optimal. These results extend to bounded-rank hypergraphs. 

We complement the algorithm with a deterministic lower bound: for every fixed polylogarithmic semi-streaming space bound, there are polylogarithmic values of $\Delta$ for which any deterministic one-pass algorithm requires $\exp(\Delta^{\Omega(1)})$ colors. 
\end{abstract}

\thispagestyle{empty}

\newpage
\hypersetup{pageanchor=true}
\pagenumbering{arabic}

\section{Introduction}
\label{sec:intro}

A hypergraph \(H=(V,E)\) is \emph{\(k\)-uniform} if every edge contains exactly \(k\) vertices. A proper coloring assigns colors to the vertices so that no edge is monochromatic. The degree of a vertex is the number of edges containing it, and \(\Delta(H)\) denotes the maximum degree. The classical Lov\'asz Local Lemma (LLL) implies that \(H\) has a proper coloring with \(O_k(\Delta(H)^{1/(k-1)})\) colors~\cite{erdos-lovasz-lll,alon-spencer}. Specifically, color each vertex independently and uniformly from \(Q\) colors. A fixed edge is monochromatic w.p. \(Q^{-(k-1)}\), and this event is independent of the colors on all disjoint edges. Each edge intersects at most \(k(\Delta(H)-1)\) other edges, so the LLL applies when \(Q\) is a sufficiently large constant multiple of \(\Delta(H)^{1/(k-1)}\).

This dependence on the maximum degree is optimal up to constant factors. Indeed, the complete \(k\)-uniform hypergraph \(K_m^{(k)}\) has maximum degree \(\binom{m-1}{k-1}=\Theta_k(m^{k-1})\). Every color class in a proper coloring has at most \(k-1\) vertices, so its chromatic number is \(\lceil m/(k-1)\rceil=\Theta_k(m)\).

Our main question is whether we can achieve the same asymptotic number of colors while seeing the edges only once and using much less space than would be needed to store the whole hypergraph. In the \emph{insertion-only streaming model}, edges arrive one at a time and are never deleted; a one-pass algorithm cannot revisit an edge unless it has stored it. We seek a \emph{semi-streaming} algorithm, whose memory is at most \(n\) times a fixed power of \(\log n\), with constants allowed to depend on the fixed value of \(k\). 
We assume that no edge appears twice and that the stream order is fixed independently of the algorithm's random choices. Our main result answers the question affirmatively.

\begin{theorem}[Streaming coloring for uniform hypergraphs]
\label{thm:uniform}
For a sufficiently large absolute constant \(A>0\) and every \(k\ge2\), there is a randomized one-pass algorithm with the following guarantee. Given integers \(n\ge2\) and \(1\le\Delta\le\max\{1,\binom{n-1}{k-1}\}\), and a stream in which every edge appears once, describing an \(n\)-vertex \(k\)-uniform hypergraph \(H\) of maximum degree at most \(\Delta\), the algorithm w.h.p. outputs a proper coloring with
\[
    Q=\left\lceil A\Delta^{1/(k-1)}\right\rceil
\]
colors.  It uses \(\widetilde O_k(n)\) bits of working memory and has expected polynomial postprocessing time.
\end{theorem}

Thus the one-pass space restriction costs no asymptotic increase in the number of colors. The same conclusion holds when edge sizes vary over a fixed bounded range. 

\begin{corollary}[Streaming coloring for bounded-rank hypergraphs]
\label{cor:bounded-rank}
Fix integers \(2\le r\le k\). Given integers \(n\ge2\) and
\(
1\le \Delta\le \max\left\{1,\sum_{q=r}^{k}\binom{n-1}{q-1}\right\},
\)
and a duplicate-free insertion stream describing an \(n\)-vertex hypergraph \(H\) of maximum degree at most \(\Delta\), where every edge size belongs to \(\{r,r+1,\ldots,k\}\), there is a randomized one-pass algorithm that w.h.p. outputs a proper coloring with
\[
O_{r,k}\!\left(\Delta^{1/(r-1)}\right)
\]
colors. It uses \(\widetilde O_{r,k}(n)\) bits of working memory and has expected polynomial postprocessing time.
\end{corollary}

Randomization is essential for obtaining so few colors within these space bounds, at least for suitable maximum degrees that are powers of \(\log n\). The next theorem says that, for any fixed semi-streaming memory bound, some such degrees force a deterministic algorithm to use exponentially more colors as a function of a positive power of \(\Delta\).

\begin{theorem}[Exponential deterministic lower bound]
\label{thm:det-streaming-lb}
Fix \(k\ge2\) and \(a\ge0\). There exist constants \(c=c(k,a)>0\) and \(C=C(k,a)>0\) such that the following holds for all sufficiently large \(n\). Set
\[
\Delta=\left\lceil C\log^{a+2}n\right\rceil.
\]
Any deterministic one-pass streaming algorithm that properly colors every \(n\)-vertex \(k\)-uniform hypergraph of maximum degree at most \(\Delta\) using \(O(n\log^a n)\) bits of space must use at least \(\exp(\Delta^c)\) colors in the worst case.
\end{theorem}

Our techniques also yield a palette sparsification theorem for general uniform hypergraphs.
Given a list \(L(v)\) of available colors at each vertex, \(H\) is \emph{\(L\)-colorable} if it has a proper coloring \(\chi\) with \(\chi(v)\in L(v)\) for every \(v\). Palette sparsification asks whether colorability is preserved when each vertex keeps only a small random subset of the global palette. We show that, from a palette of size \(\Theta(\Delta^{1/(k-1)})\), random lists of size \(\Theta(\sqrt{\log n})\) suffice for general uniform hypergraphs.

\begin{theorem}[Palette sparsification for uniform hypergraphs]
\label{thm:palette_sparsification}
For sufficiently large absolute constants \(A,K>0\), every \(k\ge2\), all sufficiently large \(n\), and every \(\Delta\ge1\), let \(H\) be an \(n\)-vertex \(k\)-uniform hypergraph of maximum degree at most \(\Delta\), and set
\[
Q=\left\lceil A\Delta^{1/(k-1)}\right\rceil, \qquad\ell=\min\left\{Q,\left\lceil K\sqrt{\log n}\right\rceil\right\}.
\]
If every vertex \(v\) independently receives a uniformly random \(\ell\)-subset \(L(v)\subseteq[Q]\), then \(H\) is \(L\)-colorable w.h.p.
\end{theorem}

\begin{remark}[Constants]
The constants \(A\) and \(K\) are independent of \(k\).  Consequently, \Cref{thm:palette_sparsification} also applies when \(k\) grows with \(n\).
\end{remark}

\begin{remark}[Tightness]\label{remark:optimality}
The construction of \cite{CE26} gives a lower bound \(\ell=\Omega_{A,k}(\sqrt{\log n})\) when \(Q=\lceil A\Delta^{1/(k-1)}\rceil\).  Consider a disjoint union of copies of \(K^{(k)}_{(k-1)\ell+1}\).  Here \(\Delta=\binom{(k-1)\ell}{k-1}\) and hence \(A\Delta^{1/(k-1)}=\Theta_{A,k}(\ell)\).  Within one component, the probability that all vertices receive the same \(\ell\)-element list is \(\exp(-\Theta_{A,k}(\ell^2))\).  On that event, the component is not list-colorable: among its \((k-1)\ell+1\) vertices, one of the \(\ell\) colors must appear on at least \(k\) vertices.  There are \(\Theta_k(n/\ell)\) independent components.  Thus, for a sufficiently small \(c_{A,k}>0\), lists of size at most \(c_{A,k}\sqrt{\log n}\) fail w.h.p.  The \(O(\sqrt{\log n})\) bound is therefore asymptotically tight.
\end{remark}

\begin{corollary}[Palette sparsification for bounded-rank hypergraphs]
\label{cor:palette_sparsification_bounded}
For sufficiently large absolute constants \(A,K>0\), integers \(2\le r\le k\), all sufficiently large \(n\), and every \(\Delta\ge1\), let \(H\) be an \(n\)-vertex hypergraph of maximum degree at most \(\Delta\) whose edge sizes belong to \(\{r,r+1,\ldots,k\}\).  Set
\[
Q=\left\lceil A\Delta^{1/(r-1)}\right\rceil,
\qquad \ell=\min\left\{Q,\left\lceil K\sqrt{\log n}\right\rceil\right\}.
\]
If every vertex \(v\) independently receives a uniformly random \(\ell\)-subset \(L(v)\subseteq[Q]\), then \(H\) is \(L\)-colorable w.h.p.
\end{corollary}

\paragraph{Related work.} Assadi, Chen, and Khanna~\cite{ACK19} introduced palette sparsification for graph coloring and used it to obtain semi-streaming algorithms. Alon and Assadi~\cite{AA20} proved, among other results, that random lists of size \(\Theta_\varepsilon(\sqrt{\log n})\) are sufficient and necessary when the global palette has size \(\lceil(1+\varepsilon)\Delta\rceil\), again with streaming consequences. More recently, Assadi and Yazdanyar~\cite{AY25} developed an asymmetric form of graph palette sparsification, in which different vertices may receive lists of different sizes. Assadi, Chen, and Sun~\cite{assadi2022deterministic} established strong deterministic one-pass lower bounds for graph coloring. Their proof relates the memory of an algorithm to the edges whose absence it can determine; our \Cref{thm:det-streaming-lb} extends this method to uniform hypergraphs.

Streaming coloring has also been studied directly for hypergraphs. Radhakrishnan, Shannigrahi, and Venkat~\cite{RSV15} considered two-coloring and gave space-efficient randomized one-pass algorithms, as well as deterministic space lower bounds, for hypergraphs with large edge sizes and suitable bounds on the number of edges. 

Adamson, Halld\'orsson, and Nolin~\cite{AHN23} gave a randomized one-pass semi-streaming algorithm for linear uniform hypergraphs, where a hypergraph is linear if any two distinct edges intersect in at most one vertex. For fixed \(k\ge3\), their algorithm uses \(O_k((\Delta/\sigma)^{1/(k-1)})\) colors, where \(\sigma=\min\{\log\Delta,\log\log n\}\), for \(\Delta\ge2\) and sufficiently large \(n\). Our \Cref{thm:uniform} applies to general uniform hypergraphs, where many edges may share several vertices.

For hypergraph palette sparsification, Casselgren~\cite{Casselgren24} showed that random lists of size \(O_{A,k}(\log n)\) suffice when the global palette has size \(\lceil A\Delta^{1/(k-1)}\log n\rceil\). Casselgren and Eriksson~\cite{CE26} proved that lists of size \(O_k((\log n)^{1/k})\) suffice for linear uniform hypergraphs with a palette of size \(O_k(\Delta^{1/(k-1)})\). Their construction implies an \(\Omega_{A,k}(\sqrt{\log n})\) lower bound for general uniform hypergraphs when the palette has size \(\lceil A\Delta^{1/(k-1)}\rceil\) (see discussion in \Cref{remark:optimality}). Our \Cref{thm:palette_sparsification} matches this lower bound, and \Cref{thm:uniform} achieves the same number of colors in one pass using semi-streaming space.

\subsection{Proof overview} The algorithm begins by giving each vertex an independent random list of \(\ell=\Theta_k(\sqrt{\log n})\) colors from a palette of size \(Q=\Theta(\Delta^{1/(k-1)})\). An arriving edge can become monochromatic only in a color common to all its vertices' lists. For each such color \(c\), we store the pair \((e,c)\) as a coloring constraint that the vertices of \(e\) must not all receive \(c\). Edges with no common available color need not be stored. After the stream, we simplify the stored constraints, remove a few problematic colors from each list, and use an algorithmic form of the LLL to obtain the final coloring.

The total expected space is \(\widetilde O_k(n)\) bits. The basic saving is already visible when \(k=3\), where \(Q=\Theta(\sqrt{\Delta})\). For an edge \(e=\{u,v,w\}\), the expected number of colors common to all three lists is \(Q(\ell/Q)^3=\widetilde O(1/\Delta)\). Since there are at most \(n\Delta/3\) edges, the expected number of stored pairs \((e,c)\) is only \(\widetilde O(n)\). A fixed memory limit and independent repetitions turn the expected space bound into a guaranteed one.

The main difficulty is that a small total number of stored constraints does not prevent many of them from sharing the same vertices. For example, if many triples contain the same pair \(\{u,v\}\), then forbidding \(u\) and \(v\) from both receiving a common color \(c\) satisfies all corresponding constraints. More generally, we call a shared proper subset \(S\) of an edge, with at least two vertices, a \emph{core}. Our algorithm samples occurrences of \emph{active} cores, whose vertices share a common available color, and uses cores with large sampled counts to generate constraints on fewer vertices. After the stream, constraints are processed from smaller sets to larger ones, discarding any constraint implied by one on a proper subset.

Previous palette-sparsification arguments use local color-degree bounds to control how many relevant edges involve a given vertex and color~\cite{AA20,CE26}. To obtain analogous control for general hypergraphs in one pass, for each vertex \(v\) we form auxiliary hypergraphs describing the possible kept constraints involving \(v\). We establish a \emph{degree hierarchy} bounding how many auxiliary edges can contain any given set of vertices. This implies that only a few colors at each vertex occur in too many kept constraints. After removing these colors, enough remain for the LLL to yield a proper coloring.

The palette-sparsification theorem uses the same probabilistic estimates and final coloring step. However, without the streaming constraint, the core compression can be carried out directly using exact incidence counts rather than sampled estimates. For the deterministic lower bound, we extend the result of Assadi, Chen, and Sun~\cite{assadi2022deterministic} to \(k\)-uniform hypergraphs.

\paragraph{Organization.} After the preliminaries, \Cref{sec:overloaded_are_rare} shows that few colors need to be removed. \Cref{sec:lll} then shows that sublists satisfying the required local constraint-count bounds admit a coloring avoiding all stored bad events.
\Cref{sec:algorithm} constructs and analyzes the streaming algorithm. \Cref{sec:compression} then proves the palette-sparsification theorem, and \Cref{sec:deterministic-lower-bound} proves the deterministic lower bound. The appendices give the detailed concentration proof and the streaming extension to bounded-rank hypergraphs.
\section{Preliminaries}
\label{sec:prelim}

\subsection{Hypergraph notation}
Unless stated otherwise, \(H=(V,E)\) is an \(n\)-vertex \(k\)-uniform hypergraph with \(k\ge2\), and \(e(H)=|E|\). We focus on the nontrivial cases \(n\ge k\) and \(E\ne\emptyset\). The codegree of a set \(S\subseteq V\), with \(|S|\le k\), counts the edges containing all vertices of \(S\):
\[
    d_H(S):=|\{e\in E:S\subseteq e\}|.
\]
Thus \(d_H(\emptyset)=e(H)\), \(d_H(\{v\})\) is the ordinary vertex degree, and \(d_H(\{u,v\})\) counts the edges containing both \(u\) and \(v\). We have \(\max_{v\in V}d_H(\{v\})=\Delta(H)\le\Delta\).

Given a list assignment \(L\), we write
\[
    L(S):=\bigcap_{v\in S}L(v)
\]
for the colors available at every vertex of \(S\). We say that \(S\) is active if \(L(S)\neq\emptyset\). A core is a proper subset of an input edge; the streaming algorithm uses cores of sizes \(2,\ldots,k-1\) to represent shared parts of edges.

A color-specific constraint \((S,c)\) requires that the vertices of \(S\) are not all assigned color \(c\). Under a random choice of vertex colors, its violation is the \emph{bad event} \(B_{S,c}\) that all these vertices receive \(c\). We use the same label \(B_{S,c}\) for the stored constraint and for this event, depending on whether we are describing the algorithm or its probability analysis. Satisfying a constraint on \(S\) also satisfies the same-color constraint on every superset of \(S\).

For a palette size \(Q\) and a list size \(\ell\), we define \(D:=Q/80\) and \(\Lambda:=\ell/40\). Powers of \(D\) will bound edge counts before sampling the lists, while powers of \(\Lambda\) will bound the number of constraints involving a fixed vertex and color afterward. In arguments using these parameters, we assume \(D>1\) and \(\Lambda>1\); the cases with small palettes are handled separately.

\subsection{Streaming model and other notation}
For a positive integer \(N\), write \([N]=\{1,\ldots,N\}\). Throughout the streaming sections, \(k\ge2\) is fixed. Each arriving edge is given by the identifiers of its \(k\) distinct vertices, and the algorithm is given \(n\), \(k\), and \(\Delta\) in advance. It sees every edge only once, in an order fixed independently of its random choices, and outputs the coloring after the stream. We call the computation performed after the last edge \emph{postprocessing}. Its working memory is included in the space bound.

Space is measured in bits. The notation \(\widetilde O_k(\cdot)\) suppresses factors that are fixed powers of \(\log n\) and constants depending on \(k\). Since the input has no repeated edges, we may assume \(\Delta\le\binom{n-1}{k-1}\). Thus a vertex identifier, color, counter, or pointer takes \(O_k(\log n)\) bits. 

\subsection{Algorithmic lopsided LLL solver}
\label{subsec:lll-solver}

The LLL provides a way to avoid many unlikely bad events simultaneously. We use an algorithmic version: choose all vertex colors independently, and whenever a bad event occurs, redraw the colors of the vertices involved in that event. 
To distinguish overlap from conflict, consider the following example. The event that \(u\) and \(v\) both receive red conflicts with the event that \(v\) and \(w\) both receive blue, because they require different colors at \(v\). Two events requiring red at \(v\) overlap but do not conflict in this sense. The lopsided LLL allows us to count only these conflicts. The following standard result states the version we need.

\begin{lemma}[Algorithmic lopsided local lemma (Theorem~6.1 of \cite{moser-tardos}; also~\cite{pegden-lopsided})]
\label{thm:lll-solver}
Let \(\Omega=\prod_{v\in V}\Omega_v\) be a product probability space with independent variables \(X_v\in\Omega_v\). Let \(\calB\) be a finite family of atomic bad events, meaning that each event specifies one value for each of the variables it involves. Write \(\vbl(B)\subseteq V\) for this set of variables, and \(\sigma_B(v)\in\Omega_v\) for the required values, so
\begin{align*}
    B=\bigcap_{v\in\vbl(B)}[X_v=\sigma_B(v)].
\end{align*}
Define the conflict graph on \(\calB\) by joining \(B\) and \(B'\) if they require different values of some common variable. Suppose we can assign a weight \(x_B\in(0,1)\) to each event such that, for every \(B\in\calB\),
\begin{equation}
\label{eq:lll-condition}
    \Prb[B]
    \le
    x_B\prod_{B'\sim B}(1-x_{B'}),
\end{equation}
where the product is over neighbors in the conflict graph. Then the lopsided Moser--Tardos algorithm, which repeatedly resamples the variables of any bad event that occurs, terminates w.p. one and avoids all events in \(\calB\). Moreover, the expected number of resamplings of event \(B\) is at most \(x_B/(1-x_B)\), and the expected total number is at most \(\sum_{B\in\calB}x_B/(1-x_B)\).
\end{lemma}

In our application, \(X_v\) is the color of vertex \(v\), and \(B_{S,c}\) requires \(X_v=c\) for every \(v\in S\). To apply the lemma, we will keep a small collection of these constraints and bound how many involve any fixed vertex and color.

\section{Concentration and LLL Tools}
This section develops two tools used in both the streaming and palette-sparsification results.
\subsection{Overloaded sampled colors are rare}\label{sec:overloaded_are_rare}

In this subsection, each \(L(v)\) is an independent uniform \(\ell\)-subset of \([Q]\).

Fix a vertex \(v\) and a size \(s\). We will describe possible constraints on \(s\)-vertex sets containing \(v\) by an auxiliary \((s-1)\)-uniform hypergraph \(F_{s,v}\) on \(V\setminus\{v\}\): an edge \(U\) represents the set \(\{v\}\cup U\). The precise construction will differ between the streaming version and the offline version, but in each case \(F_{s,v}\) is fixed before the lists are sampled. For a color \(c\), keep only auxiliary edges whose vertices all have \(c\) in their lists:
\[
    E(F_{s,v,c}):=\{e\in E(F_{s,v}):c\in L(e)\}.
\]

Thus \(F_{s,v,c}\) is the subhypergraph induced by the vertices that have color \(c\) available. If \(c\in L(v)\), its edges describe possible constraints involving \(v\) and \(c\). Bounding \(e(F_{s,v,c})\) will therefore bound the number of such constraints.

We call a sampled color \(c\in L(v)\) \emph{overloaded} for \(v\) if
\[
e(F_{s,v,c})>\Lambda^{s-1}
\qquad\text{for some }2\le s\le k,
\]
and write
\[
\widetilde{\mathrm{Bad}}(v):=\{c\in L(v):c\text{ is overloaded for }v\}.
\]

\begin{assumption}[Degree hierarchy]\label{ass:residual-hierarchy}
For every vertex \(v\) and every \(2\le s\le k\), there is an \((s-1)\)-uniform hypergraph \(F_{s,v}\) on \(V\setminus\{v\}\), fixed independently of the sampled lists, such that
\[
e(F_{s,v})\le D^{s-1},
\]
and, for every \(1\le j\le s-1\),
\[
\max_{|T|=j}d_{F_{s,v}}(T)\le \Gamma D^{s-1-j},
\]
where \(\Gamma\ge1\).
\end{assumption}


\begin{lemma}[One color is rarely overloaded]\label{lem:literal-loads}
Under Assumption~\ref{ass:residual-hierarchy}, for every \(v\in V\) and \(c\in[Q]\), and for sufficiently large \(\ell\) depending on \(\Gamma\),
\[
\Prb\left[c\in\widetilde{\mathrm{Bad}}(v)\,\middle|\,c\in L(v)\right]\le e^{-\Omega_{\Gamma}(\ell)}.
\]
\end{lemma}

\begin{proof}
The proof appears in Appendix~\ref{sec:proof-of-concentration}. It bounds a high moment of the edge count \(e(F_{s,v,c})\), using the degree hierarchy to control overlapping edges, and then applies Markov's inequality.
\end{proof}

We next bound the number of overloaded colors in each list. The list is sampled without replacement, so the availability of different colors is not independent. We use \emph{negative association}, which gives the needed substitute for independence.

\begin{lemma}[Few colors are overloaded]\label{lem:palette-trimming}
Under Assumption~\ref{ass:residual-hierarchy}, let
\(\ell=\left\lceil K_\Gamma\sqrt{\log n}\right\rceil\).
For sufficiently large \(K_\Gamma\) depending on \(\Gamma\), and all sufficiently large \(n\), w.p. at least \(1-n^{-\Omega_\Gamma(1)}\), simultaneously for every vertex \(v\),
\[
|L(v)\setminus\widetilde{\mathrm{Bad}}(v)|\ge \lfloor\ell/2\rfloor.
\]
\end{lemma}

\begin{proof}
Fix a vertex \(v\) and condition on \(L(v)\). For \(u\neq v\) and \(c\in[Q]\), let \(Y_{u,c}:=\mathbf1[c\in L(u)]\). For each fixed \(u\), the variables \(\{Y_{u,c}\}_{c\in[Q]}\) are negatively associated, since \(L(u)\) is a uniformly random \(\ell\)-subset of \([Q]\). As the lists are independent across vertices, the collection \(\{Y_{u,c}:u\neq v,\ c\in[Q]\}\) is also negatively associated~\cite{JoagDevProschan83}.

For a fixed color \(c\), overload means that \(e(F_{s,v,c})>\Lambda^{s-1}\) for some \(s\). This event is increasing: adding \(c\) to more vertex lists can only add edges to \(F_{s,v,c}\). It depends only on the variables \(\{Y_{u,c}:u\neq v\}\), so different colors depend on disjoint groups of variables. Hence, for any fixed \(t\) distinct colors of \(L(v)\), negative association and \Cref{lem:literal-loads} give
\[
\Prb[\text{all \(t\) colors are overloaded}]\le e^{-\Omega_\Gamma(\ell t)}.
\]
If at least half the list is overloaded, some specified set of \(\lceil\ell/2\rceil\) colors must all be overloaded. A union bound over these choices gives
\[
\Prb[|\widetilde{\mathrm{Bad}}(v)|\ge\lceil \ell/2\rceil]\le\binom{\ell}{\lceil \ell/2\rceil}e^{-\Omega_\Gamma(\ell \lceil \ell/2\rceil)}=e^{-\Omega_\Gamma(\ell^2)}.
\]
Since \(\ell=\lceil K_\Gamma\sqrt{\log n}\rceil\), choosing \(K_\Gamma\) sufficiently large and taking a union bound over \(v\in V\) gives
\[
\Prb\left[|L(v)\setminus\widetilde{\mathrm{Bad}}(v)|\ge\lfloor \ell/2\rfloor\text{ for all }v\in V\right]\ge1-n^{-\Omega_\Gamma(1)}.
\qedhere
\]
\end{proof}

\subsection{Local Lemma completion}\label{sec:lll}

Given a family \(\mathcal B=\{B_{S,c}\}\) of color-specific constraints with \(c\in L(S)\), define the \emph{local constraint count}
\[
    N_s(v,c):=\left|\left\{B_{S,c}\in\mathcal B:|S|=s,\ v\in S\right\}\right|,
\]
for every \(v\in V\), \(c\in[Q]\), and \(2\le s\le k\). Thus \(N_s(v,c)\) counts constraints on exactly \(s\) vertices that include \(v\) and forbid color \(c\). The problematic colors we remove at \(v\) are
\[
\mathrm{Bad}(v):=\{c\in L(v):N_s(v,c)>\Lambda^{s-1}
\text{, for some }2\le s\le k\}.
\]

Given a list $G(v)$ for every vertex $v$, let $X_v$ be chosen independently and uniformly from $G(v)$. Each $B_{S,c}\in\mathcal B$ with $c\in G(v)$ for every $v\in S$ is then realized as the bad event
\[
    B_{S,c}=\{X_v=c\text{ for every }v\in S\}.
\]
Let \(\mathcal B_G\) denote the family of all such events. A constraint whose color is absent from one of its vertices' new lists is automatically satisfied, so it does not belong to \(\mathcal B_G\). Unlike \(\widetilde{\mathrm{Bad}}(v)\), which was defined using an auxiliary hypergraph, \(\mathrm{Bad}(v)\) uses the constraints actually kept. 

\begin{lemma}[LLL completion]\label{lem:lll-verification} 
Suppose that \(|G(v)|=m\) for every vertex \(v\), and that for every \(v\in V\), every \(c\in G(v)\), and every \(2\le s\le k\),
\[
    N_s(v,c)\le\Lambda^{s-1}.
\]
If \(\Lambda\le \frac{m}{10}\), then, for sufficiently large \(m\), the hypotheses of \Cref{thm:lll-solver} hold for \(\calB_G\). Consequently, \(\textnormal{\textsc{LLL-Solver}}\) finds an assignment avoiding all events in \(\calB_G\). 
\end{lemma}

\begin{proof}
An event on \(|S|\) vertices has probability \(m^{-|S|}\). We give it a slightly larger weight to allow for its conflicts with other events: for every \(B_{S,c}\in\calB_G\), set
\[ 
x_{S,c}:=\left(\frac{2}{m}\right)^{|S|}. 
\] 
Since every event involves at least two vertices, \(\max_{B\in\calB_G}x_B=o(1)\) as \(m\) grows. Fix \(B_{S,c}\in\calB_G\). An event in \(\calB_G\) conflicting through a vertex \(v\in S\) must require a color \(c'\in G(v)\setminus\{c\}\) at \(v\). There are at most \(N_s(v,c')\le\Lambda^{s-1}\) such events of size \(s\). Therefore the sum of their weights is at most
\[ 
\sum_{c'\in G(v)\setminus\{c\}}\sum_{s=2}^k \Lambda^{s-1}\left(\frac{2}{m}\right)^s \le m\sum_{s=2}^k \Lambda^{s-1}\left(\frac{2}{m}\right)^s \le \frac{4\Lambda/m}{1-2\Lambda/m} \le \frac12. 
\] 
Summing over \(v\in S\) may count an event more than once, but still gives the upper bound
\[ 
\sum_{B\sim B_{S,c}}x_B\le \frac{|S|}{2}. 
\] 
For small weights, the product of \(1-x_B\) is bounded below by the exponential of approximately minus their sum. Thus, for sufficiently large \(m\),
\[ 
x_{S,c}\prod_{B\sim B_{S,c}}(1-x_B)\ge \left(\frac{2}{m}\right)^{|S|}\exp\left(-(1+o(1))\frac{|S|}{2}\right)\ge m^{-|S|}=\Prb[B_{S,c}]\;. 
\] 
Thus the lopsided Local Lemma condition holds for every event in \(\calB_G\), and \Cref{thm:lll-solver} applies. 
\end{proof}


\section{The streaming algorithm for uniform hypergraph coloring}
\label{sec:algorithm}

This section proves the main streaming result, \Cref{thm:uniform}. The input is an \(n\)-vertex \(k\)-uniform hypergraph \(H=(V,E)\) of maximum degree at most \(\Delta\), with no repeated edges and fixed \(k\ge2\). We store only constraints for colors common to the vertices of an edge. At the same time, random sampling identifies small vertex sets contained in many edges, which are then used to replace larger constraints.

Fix
\begin{align}\label{eq:definition_of_Q}
    Q:=\left\lceil A\Delta^{1/(k-1)}\right\rceil,\qquad \ell:=\left\lceil K_k\sqrt{\log n}\right\rceil.
\end{align}
We first choose the absolute constant \(A\) sufficiently large. For each fixed \(k\), we then choose \(K_k\) large enough for the probability estimate in \Cref{lem:literal-loads} and \(R_k\) large enough for the sampling guarantee in \Cref{lem:active-core-detection}. The constant \(K_k\) determines the list size, and \(R_k\) determines how often we count a core occurrence.

\begin{algorithm}[!htbp]
\caption{A streaming algorithm for uniform hypergraph coloring}
\label{alg:hypergraph-kernel}
\begin{algorithmic}[1]
\Require An insertion-only one-pass stream of \(n\)-vertex \(k\)-uniform edges, degree bound \(\Delta\), palette size \(Q\), list size \(\ell\), thresholds \(Q^{k-s}\), and sampling constant \(R_k\).
\Ensure A coloring \(\chi:V\to[Q]\), or \textsc{Fail}.
\State Independently sample a list \(L(v)\subseteq[Q]\) of size \(\ell\) for all vertices \(v\in V\).
\State Initialize an empty list \(\mathcal K_{\mathrm{edge}}\) of pairs \((e,c)\); counters \(C_S\) are initially zero and stored only when positive.
\State Write \(B_{S,c}\) for the constraint that the vertices of \(S\) not all receive color \(c\).
\For{each arriving edge \(e\)}
    \State Compute \(L(e)=\bigcap_{v\in e}L(v)\).
    \For{each color \(c\in L(e)\)}
        \State Append the full-edge constraint \((e,c)\) to \(\mathcal K_{\mathrm{edge}}\).
    \EndFor
    \For{each subset \(S\subsetneq e\) with \(2\le |S|=s\le k-1\)}
        \State Compute \(L(S)=\bigcap_{v\in S}L(v)\).
        \If{\(L(S)\neq\emptyset\)}
            \State w.p. \(\rho_s:=R_k\log n/Q^{k-s}\), increment the sparse counter \(C_S\), capped at \(\lceil R_k\log n\rceil\).
        \EndIf
    \EndFor
\EndFor
\For{\(s=2,3,\ldots,k-1\)}
    \State \(\widehat{\mathcal P}_s\gets \{S: |S|=s,\ C_S\ge (R_k/2)\log n\}\).
\EndFor
\State Let \(\mathcal K\) be the set consisting of all \(B_{S,c}\) with \(S\in\widehat{\mathcal P}_s\), \(2\le s\le k-1\), and \(c\in L(S)\), together with all \(B_{e,c}\) corresponding to \((e,c)\in\mathcal K_{\mathrm{edge}}\).
\State \(\mathcal B\gets\emptyset\).
\For{\(s=2,3,\ldots,k\)}
    \For{each \(B_{S,c}\in\mathcal K\) with \(|S|=s\)}
        \If{there is no \(B_{T,c}\in\mathcal B\) with \(T\subsetneq S\)}
            \State Add \(B_{S,c}\) to \(\mathcal B\).
        \EndIf
    \EndFor
\EndFor
\State Compute all \(N_s(v,c):=\left|\left\{B_{S,c}\in\mathcal B:|S|=s,\ v\in S\right\}\right|\).
\State Set
\(
    L'(v)\gets L(v)\setminus \{c:\exists s,\ N_s(v,c)>\Lambda^{s-1}\}\text{, for every }v\in V.
\)
\If{some \(|L'(v)|<\lfloor\ell/2\rfloor\)}
    \State Output \textsc{Fail}.
\EndIf
\State For every \(v\in V\), choose an arbitrary sublist \(G(v)\subseteq L'(v)\) of size \(m:=\lfloor\ell/2\rfloor\).
\State Let \(\mathcal B_G\) be obtained from \(\mathcal B\) by deleting every constraint \(B_{S,c}\) for which \(c\notin G(S)\).
\State Run \(\textsc{LLL-Solver}(\mathcal B_G,\{G(v)\}_{v\in V})\).
\If{the solver returns an assignment \(\chi(v)\in G(v)\) avoiding all events in \(\mathcal B_G\)}
    \State Output \(\chi\).
\Else
    \State Output \textsc{Fail}.
\EndIf
\end{algorithmic}
\end{algorithm}

\paragraph{Overview of the algorithm.} We first consider \(Q\ge2\ell\) and \(Q^{k-s}\ge R_k\log n\) for every \(2\le s\le k-1\). These inequalities ensure that the palette is larger than the sampled lists and that all sampling probabilities below are at most one. In the remaining cases, the maximum degree is small enough to store the entire input; we handle them in \Cref{sec:proof-uniform}.

Each vertex \(v\) independently samples a uniform \(\ell\)-subset \(L(v)\subseteq[Q]\). When an edge \(e\) arrives, we store \((e,c)\) for every \(c\in L(e)\). This pair records the constraint that \(e\) not be monochromatic in \(c\); we call it a \emph{full-edge constraint} because it involves all \(k\) vertices.

For each active subset \(S\subsetneq e\) of size \(s\ge2\), the arrival of \(e\) is one \emph{occurrence} of the core \(S\). We count this occurrence independently w.p.
\[
\rho_s:=R_k\log n/Q^{k-s}
\]
using a counter \(C_S\). A counter is stored only after its first counted occurrence, so we do not allocate space for all possible vertex subsets. 
If \(S\) belongs to roughly \(Q^{k-s}\) edges, its expected sampled count is roughly \(R_k\log n\), which is enough to detect it reliably (see \Cref{lem:active-core-detection}).

After the stream, \(\widehat{\mathcal P}_s\) are the \(s\)-vertex sets with sufficiently large sampled counts. For each such \(S\) and each \(c\in L(S)\), we add the constraint that \(S\) not be monochromatic in \(c\). We process these constraints, together with the full-edge constraints, in increasing order of the number of vertices involved. We keep \(B_{S,c}\) only if no already kept \(B_{T,c}\) has \(T\subsetneq S\). Thus, for each color, no kept set contains another kept set. Write \(\mathcal B\) for the resulting family.

We next remove \(\mathrm{Bad}(v)\), the colors involved in too many kept constraints at \(v\), and arbitrarily choose a sublist \(G(v)\) of size \(m=\lfloor\ell/2\rfloor\). If some vertex has too few colors left, this copy of the algorithm reports failure; \Cref{cor:palette-trimming} shows that this is unlikely. Otherwise, let \(\mathcal B_G:=\{B_{S,c}\in\mathcal B:c\in G(S)\}\) be the constraints that can still be violated under the new lists. Finally, \Cref{lem:lll-verification} finds an assignment avoiding all constraints in \(\mathcal B_G\).

\Cref{alg:hypergraph-kernel} gives the pseudocode for one copy of the algorithm. In the proof of \Cref{thm:uniform}, we cap its memory and run independent copies to obtain the high-probability space guarantee. The remaining small-degree cases are handled separately by storing the entire input.


\subsection{Correctness and expected space}
\label{sec:correctness-space}

We first prove the two basic guarantees of the stored information. 

\begin{lemma}[Correctness]
\label{lem:kernel-correct}
If an assignment \(\chi(v)\in G(v)\) avoids all events in \(\calB_G\), then \(\chi\) is a proper coloring of the original hypergraph.
\end{lemma}

\begin{proof}
Suppose, for a contradiction, that an original edge \(e\) is monochromatic in color \(c\). Because \(\chi(v)\in G(v)\subseteq L(v)\), we have \(c\in G(e)\subseteq L(e)\), so the pair \((e,c)\) was stored when \(e\) arrived. If \(B_{e,c}\) was kept, it belongs to \(\calB_G\) and occurs. If it was discarded, an already kept \(B_{S,c}\) with \(S\subsetneq e\) made it unnecessary. All vertices of \(S\) also receive \(c\), and \(c\in G(S)\), so this smaller bad event belongs to \(\calB_G\) and occurs. Either case contradicts the assumption.
\end{proof}

\begin{lemma}[Expected space bound]
\label{lem:space}
One copy of the streaming algorithm uses \(\widetilde O_k(n)\) bits in expectation over the random lists and the random decisions to count core occurrences.
\end{lemma}

\begin{proof}
Each stored set has at most \(k\) vertex identifiers and one color or counter, so it takes \(O_k(\log n)\) bits. The lists themselves use \(O(n\ell\log Q)=\widetilde O_k(n)\) bits. It remains to count the full-edge constraints and sampled counters. For a fixed edge and color, all \(k\) lists contain that color w.p. \((\ell/Q)^k\). Summing over edges and colors gives
\[
|E|\cdot \E|L(e)|\le \frac{n\Delta}{k}\cdot Q\left(\frac{\ell}{Q}\right)^k=O_k(n\ell^k)=\widetilde O_k(n).
\]

For a fixed \(s\)-vertex subset, a union bound over colors shows that it is active w.p. at most \(Q(\ell/Q)^s=\ell^s/Q^{s-1}\). Each edge has \(\binom{k}{s}\) such subsets, and an active occurrence is counted w.p. \(\rho_s\). Thus the expected number of counted active occurrences is at most
\[
\binom{k}{s}|E|\frac{\ell^s}{Q^{s-1}}\frac{R_k\log n}{Q^{k-s}}=O_k(n\ell^s\log n)=\widetilde O_k(n).
\]
Since each \(S\in\widehat{\mathcal P}_s\) requires at least \((R_k/2)\log n\) sampled occurrences and creates at most \(\ell\) color-specific events, the expected number of candidate and kept core events is also \(\widetilde O_k(n)\). Summing over \(2\le s\le k-1\), and multiplying the number of stored objects by \(O_k(\log n)\) bits per object, gives the claimed \(\widetilde O_k(n)\) bound.

There is no extra cost for uncounted cores: a counter is created only when an occurrence is counted. Candidate and kept constraints are stored as lists of sets and colors, so the same bit bound covers them as well.
\end{proof}

\subsection{Bounding the local constraint counts}
\label{sec:loads}

We now bound how many kept constraints can involve a fixed vertex \(v\) and color \(c\). Recall that
\[
    N_s(v,c):=|\left\{B_{S,c}\in \calB: |S|=s,v\in S\right\}|.
\]
For \(s<k\), these are constraints on cores; for \(s=k\), they come from input edges. Our goal is \Cref{lem:domination}, which bounds \(N_s(v,c)\) by an edge count in an auxiliary hypergraph. This will allow us to apply the concentration results from \Cref{sec:overloaded_are_rare}.

\begin{lemma}
\label{lem:active-core-detection}
Fix the sampled list assignment \(L\). W.p. at least \(1-n^{-20}\) over the decisions to count core occurrences, provided \(R_k\) is sufficiently large as a function of \(k\), the following hold simultaneously for every \(2\le s\le k-1\) and every \(s\)-set \(S\):
\begin{enumerate}[label=(\roman*)]
    \item If \(L(S)\neq\emptyset\) and \(d_H(S)\ge Q^{k-s}\), then \(S\in\widehat{\calP}_s\).
    \item If \(S\in\widehat{\calP}_s\), then \(d_H(S)>Q^{k-s}/4\).
\end{enumerate}
\end{lemma}

\begin{proof}
For a fixed active \(s\)-set \(S\), each of the \(d_H(S)\) edges containing \(S\) gives an independent chance to increment its counter. Thus the number counted before imposing the cap is \(Z_S\sim\operatorname{Bin}(d_H(S),\rho_s)\). The stored value is \(C_S=\min\{Z_S,\lceil R_k\log n\rceil\}\). Since the cap is above the candidate threshold, it does not affect whether \(S\) is declared a candidate.

If \(d_H(S)\ge Q^{k-s}\), then \(\E Z_S\ge R_k\log n\), twice the candidate threshold. A Chernoff bound gives \(C_S<(R_k/2)\log n\) w.p. \(\exp(-\Omega(R_k\log n))\). If \(d_H(S)\le Q^{k-s}/4\), then \(\E Z_S\le(R_k/4)\log n\), half the threshold, and a Chernoff bound gives the same upper bound on the probability of incorrectly declaring \(S\) a candidate. There are at most \(\sum_{s=2}^{k-1}n^s\le kn^k\) sets to consider. A union bound proves the lemma when \(R_k\) is sufficiently large.
\end{proof}
Let \(\mathcal E_{\mathrm{det}}\) be the good event from \Cref{lem:active-core-detection}.

The constraints kept by the algorithm depend on random choices. To analyze them, we define a larger family that contains every kept constraint whenever detection succeeds. This family is described by auxiliary hypergraphs constructed only from the input hypergraph. Their thresholds allow the constant-factor uncertainty in the sampled counters.

\paragraph{Auxiliary hypergraphs with relaxed thresholds.} Fix \(v\). An auxiliary edge \(U\) represents a possible kept constraint on \(\{v\}\cup U\). On \(\mathcal E_{\mathrm{det}}\), two conditions are necessary: this set must belong to enough input edges to have been detected, and no smaller shared set containing \(v\) can force its removal. Formally, for \(2\le s\le k-1\), a set \(U\subseteq V\setminus\{v\}\) of size \(s-1\) is an edge of \(F_{s,v}\) if
\begin{enumerate}[label=(\alph*)]
    \item \(d_H(\{v\}\cup U)>Q^{k-s}/4\), and
    \item for every nonempty \(T\subsetneq U\), writing \(|T|=a<s-1\),
    \begin{align}\label{eq:hierarchy_bound}
        d_H(\{v\}\cup T)<Q^{k-a-1}.
    \end{align}
\end{enumerate}
For \(s=k\), start with the hypergraph whose edges are the sets \(e\setminus\{v\}\) for input edges \(e\) containing \(v\). Keep only those sets \(U\) satisfying condition (b), and call the result \(F_{k,v}\). Condition (a) holds automatically because \(d_H(e)=1>1/4\) for an input edge.
\begin{remark}
    The hypergraphs \(F_{s,v}\) depend only on \(H\) and the thresholds, not on the sampled lists or counters. This independence is essential when applying \Cref{lem:palette-trimming}.
\end{remark}

The next lemma justifies that every kept constraint counted by \(N_s(v,c)\) can be mapped to a distinct edge of \(F_{s,v,c}\), by removing $v$.

\begin{lemma}
\label{lem:domination}
On the event ${\mathcal E_{\mathrm{det}}}$, for every \(2\le s\le k\),
\begin{align*}
    N_s(v,c)\le e(F_{s,v,c}).
\end{align*}
\end{lemma}
\begin{proof}
First let \(s<k\), and consider a kept constraint \(B_{S,c}\) with \(v\in S\). Write \(S=\{v\}\cup U\). Since \(S\) was a candidate, \(\mathcal E_{\mathrm{det}}\) implies \(d_H(S)>Q^{k-s}/4\), which is condition (a).

To check condition (b), suppose a smaller set \(T\subsetneq S\), with \(v\in T\) and \(|T|\ge2\), had \(d_H(T)\ge Q^{k-|T|}\). The color \(c\) is common to \(S\), so it is also common to \(T\). On \(\mathcal E_{\mathrm{det}}\), the set \(T\) would therefore be detected. Its constraint \(B_{T,c}\), or a still smaller one making it unnecessary, would be processed before \(B_{S,c}\). This would cause \(B_{S,c}\) to be discarded, a contradiction. Hence condition (b) holds. Finally, \(c\in L(U)\), so \(U\in E(F_{s,v,c})\).

For \(s=k\), the same argument applies to a kept constraint \(B_{e,c}\): writing \(e=\{v\}\cup U\), condition (a) is automatic, condition (b) follows from the same smaller-set argument, and \(c\in L(e)\) gives \(U\in E(F_{k,v,c})\). Distinct constraints can be mapped to distinct edges, proving the bound.
\end{proof}

\subsection{Proof of the streaming theorem}
\label{sec:proof-uniform}

\begin{lemma}[Degree hierarchy for the auxiliary hypergraphs]
\label{lem:streaming-residual-hierarchy}
If \(A\ge162\), then the auxiliary hypergraphs defined above satisfy, for every \(k\ge2\), every \(v\in V\), and every \(2\le s\le k\),
\[
e(F_{s,v})\le D^{s-1},
\]
and, for \(1\le a\le s-1\),
\[
    \max_{|T|=a} d_{F_{s,v}}(T)
    \le 4\cdot 81^{k-2} D^{s-1-a}.
\]
\end{lemma}

\begin{proof}
For \(k=2\), necessarily \(s=2\). In this case,
\[
e(F_{2,v})=d_H(v)\le \Delta\le Q/A\le D
\]
whenever \(A\ge 80\). Moreover, for every singleton \(T\), we trivially have
\[
d_{F_{2,v}}(T)\le 1< 4\cdot 81^{k-2}D^{s-1-|T|}.
\]
Now assume \(k\ge3\). A set of size \(s-1\) belongs to at most one edge of an \((s-1)\)-uniform hypergraph, so the codegree bound is immediate when \(a=s-1\). We therefore consider \(a\le s-2\).

First we bound the total number of auxiliary edges. Every \(U\in E(F_{s,v})\) extends with \(v\) to a set contained in more than \(Q^{k-s}/4\) input edges. Conversely, an input edge containing \(v\) has exactly \(\binom{k-1}{s-1}\) subsets of size \(s\) containing \(v\). Counting these incidences in the two orders gives
\begin{align*}
    e(F_{s,v})\cdot Q^{k-s}/4&<\sum_{S\ni v,|S|=s}d_H(S)=\binom{k-1}{s-1}d_H(\{v\})\le\binom{k-1}{s-1}\Delta \tag{double counting}\\
    &\le \binom{k-1}{s-1}\left(\frac{Q}{A}\right)^{k-1}=\binom{k-1}{s-1}\left(\frac{Q}{D}\right)^{s-1}\frac{Q^{k-s}D^{s-1}}{A^{k-1}}\tag{\Cref{eq:definition_of_Q}}\\
    &\le \left(1+\frac{Q}{D}\right)^{k-1}\frac{Q^{k-s}D^{s-1}}{A^{k-1}}\tag{Binomial theorem}\\
    &=\left(\frac{81}{A}\right)^{k-1}Q^{k-s}D^{s-1}\le Q^{k-s} D^{s-1}/4.
\end{align*}
Next fix a set \(T\subseteq V\setminus\{v\}\) of size \(1\le a\le s-2\). If no auxiliary edge contains \(T\), its codegree is zero. Otherwise condition (b) bounds \(d_H(\{v\}\cup T)\). Repeating the same count, now only for sets containing \(\{v\}\cup T\), gives
\begin{align*}
    d_{F_{s,v}}(T)\cdot Q^{k-s}/4&<\sum_{S\supseteq \{v\}\cup T,|S|=s}d_H(S)=\binom{k-a-1}{s-a-1}d_H(\{v\}\cup T)\tag{double counting}\\
    &<\binom{k-a-1}{s-a-1}Q^{k-a-1}=\binom{k-a-1}{s-a-1}
    \left(\frac{Q}{D}\right)^{s-a-1}
    Q^{k-s}D^{s-a-1}\tag{\Cref{eq:hierarchy_bound}}\\
    &\le\left(1+\frac{Q}{D}\right)^{k-a-1}
    Q^{k-s}D^{s-a-1}\tag{Binomial theorem}\\
    &\le 81^{k-2}Q^{k-s}D^{s-a-1}.
    \qedhere
\end{align*}
\end{proof}

The auxiliary hypergraphs satisfy the desired degree hierarchy, and their color-specific edge counts bound the local constraint counts. Combining these facts shows that the colors removed by the algorithm form a subset of the overloaded colors, and hence at least half of every list remains.

\begin{corollary}
\label{cor:palette-trimming}
    For a sufficiently large \(K_k\) depending on \(k\), w.p. at least \(1-n^{-10}\) over the sampled lists and random core counts, simultaneously for every vertex \(v\),
    \[
    |L(v)\setminus\mathrm{Bad}(v)|\ge \lfloor\ell/2\rfloor.
    \]
\end{corollary}
\begin{proof}
By \Cref{lem:streaming-residual-hierarchy}, the deterministic auxiliary hypergraphs $F_{s,v}$ satisfy Assumption~\ref{ass:residual-hierarchy} with $\Gamma=4\cdot 81^{k-2}$. Hence, by \Cref{lem:palette-trimming}, w.p. at least $1-n^{-\Omega_k(1)}$ over the sampled lists,
\[
    {\mathcal E_{\mathrm{trim}}}:=\left\{|L(v)\setminus\widetilde{\mathrm{Bad}}(v)|\ge \lfloor\ell/2\rfloor\text{ for every }v\in V\right\}
\]
holds. The event that enough colors remain and the detection event need not be independent. We instead use the guarantee that detection succeeds for every fixed list assignment \(L\):
\[
    \Prb({\mathcal E_{\mathrm{det}}}\mid L)\ge 1-n^{-20}.
\]
Moreover, on the event ${\mathcal E_{\mathrm{det}}}$, \Cref{lem:domination} implies that $\mathrm{Bad}(v)\subseteq\widetilde{\mathrm{Bad}}(v)$ for every $v\in V$. Therefore,
\[
    \Prb({\mathcal E_{\mathrm{trim}}}\land{\mathcal E_{\mathrm{det}}})=\mathbb E_L\!\left[\mathbf 1_{{\mathcal E_{\mathrm{trim}}}}\Prb({\mathcal E_{\mathrm{det}}}\mid L)\right]\ge (1-n^{-20})\Prb({\mathcal E_{\mathrm{trim}}})\ge (1-n^{-20})(1-n^{-\Omega_k(1)})\ge 1-n^{-10},
\]
where the last inequality holds when $K_k$ is sufficiently large. On ${\mathcal E_{\mathrm{trim}}}\land{\mathcal E_{\mathrm{det}}}$, for every $v\in V$,
\[
    |L(v)\setminus\mathrm{Bad}(v)|\ge |L(v)\setminus\widetilde{\mathrm{Bad}}(v)|\ge \lfloor\ell/2\rfloor.
    \qedhere
\]
\end{proof}


\begin{proof}[Proof of \Cref{thm:uniform}]
If \(Q<2\ell\), then
\[ 
\Delta=O(\ell^{k-1})=\polylog n;
\] 
If \(Q^{k-s}<R_k\log n\) for some \(2\le s\le k-1\), then 
\[ 
Q<\left(R_k\log n\right)^{1/(k-s)}=O_k\bigl((\log n)^{1/(k-s)}\bigr), 
\]
and hence again \(\Delta=\polylog n\). In both corner cases, the entire hypergraph can be stored in \(\widetilde O_k(n)\) bits. Then a standard LLL argument, together with Moser--Tardos resampling, directly gives a proper \(Q\)-coloring. During postprocessing, the stored hypergraph is supplemented only by the current coloring and scan indices, which require \(O(n\log Q)+O_k(\log n)=\widetilde O_k(n)\) additional bits.

In the remaining case, \Cref{cor:palette-trimming} guarantees w.p. at least \(1-n^{-10}\) that every vertex has at least \(\lfloor\ell/2\rfloor\) colors outside \(\mathrm{Bad}(v)\). Choose any \(m=\lfloor\ell/2\rfloor\) of these colors as \(G(v)\). Every color $c\in G(v)$ satisfies \(N_s(v,c)\le\Lambda^{s-1}\) for all \(s\). Since \(\Lambda=\ell/40\le m/10\), \Cref{lem:lll-verification} applies: \(\textsc{LLL-Solver}\) finds an assignment avoiding every event in \(\calB_G\). By \Cref{lem:kernel-correct}, this is a proper coloring of \(H\).

The space estimate so far is in expectation. Run \(\lceil20\log n\rceil\) independent copies of \Cref{alg:hypergraph-kernel} on the same stream, and stop a copy if its memory would exceed four times the expected-space upper bound from \Cref{lem:space}. Markov's inequality bounds the probability of stopping a copy by \(1/4\). The probability that its lists lose too many colors is at most \(n^{-10}\). Thus each copy stays within its memory limit and has enough colors left w.p. at least \(1/2\). Independence implies that at least one such copy exists w.p. at least \(1-n^{-10}\). The extra factor of \(O(\log n)\) copies still leaves total space \(\widetilde O_k(n)\).

We process the surviving copies one at a time and reuse one array for the current coloring. The constraints can be simplified and their counts computed by repeated scans of the stored data. During resampling, another scan finds a violated constraint; no graph of conflicts needs to be stored. Apart from the capped copies, the coloring array and a constant number of counters and scan pointers use \(O(n\log Q)+O_k(\log n)=\widetilde O_k(n)\) additional bits. Each scan takes time polynomial in the capped copy size, and \Cref{thm:lll-solver} gives an expected \(O(|\calB_G|)=\widetilde O_k(n)\) resamplings. The expected postprocessing time is therefore polynomial.
\end{proof}

For edges of different sizes, we run the same construction separately for each size and combine the stored constraints before coloring. Appendix~\ref{sec:bounded_rank_proof} checks that the space and probability bounds continue to hold.

\begin{proof}[Proof of \Cref{cor:bounded-rank}]
    See Appendix~\ref{sec:bounded_rank_proof}.
\end{proof}

\section{Palette sparsification for uniform hypergraphs}\label{sec:compression}

This section proves \Cref{thm:palette_sparsification}: short independent random lists preserve colorability even without restrictions on how the input edges overlap. We reuse the concentration and LLL completion from \Cref{sec:overloaded_are_rare,sec:lll}. Here there is no streaming restriction, so we can inspect all edges at once. Instead of estimating which small sets belong to many edges, we find those sets exactly.

\subsection{Deterministic constraint compression}

We construct an auxiliary hypergraph whose edges may have different sizes. Initially its edges are the input edges. Whenever many current edges contain the same smaller set \(T\), we replace all edges containing \(T\) by \(T\) itself. 
The following definition specifies when to make a compression. 

\begin{definition}[Compressed constraints sequence]
We recursively construct a sequence of collections of nonempty subsets of $V$, $\{\cA_0,\cA_1,\ldots\}$, starting from $\cA_0:=E$. Given $\cA_i$, if there exist a nonempty set $T_i\subseteq V$, with $|T_i|=t_i\ge 1$, and an integer $s_i>t_i$ such that 
\begin{align}\label{eq:nonsmooth_constraint}
    |\{S\in\cA_i:T_i\subseteq S,\ |S|=s_i\}|>D^{s_i-t_i},
\end{align}
then define
\begin{align}\label{eq:compress_operation}
    \cA_{i+1}:=
    \left(\cA_i\setminus\{S\in\cA_i:T_i\subseteq S\}\right)\cup\{T_i\}.
\end{align}
\end{definition}

We now prove several properties. The first property says that the current family is always an \emph{antichain}: no two distinct sets contain one another. It prevents us from keeping an edge whose coloring constraint is already implied by a smaller edge.

\begin{lemma}\label{lem:no_antichain}
For every $i\in\mathbb N$, if $\cA_i$ exists, then $\cA_i$ is an antichain; that is, for all distinct $S_1,S_2\in\cA_i$, we have $S_1\not\subseteq S_2$.
\end{lemma}
\begin{proof}
Since $H$ is $k$-uniform and has no repeated edges, $\cA_0$ is an antichain. Suppose $\cA_i$ is an antichain and $\cA_{i+1}$ exists. The only newly inserted set is $T_i$. The family $\{S\in\cA_i:T_i\subseteq S\}$ is nonempty, otherwise \Cref{eq:nonsmooth_constraint} does not hold. Since $\cA_i$ is an antichain, there is no $S'\in\cA_i$ with $S'\subseteq T_i$. Moreover, every set of $\cA_i$ containing $T_i$ is removed. Hence $\cA_{i+1}$ is again an antichain. The claim follows by induction.
\end{proof}

\begin{lemma}[Final degree hierarchy]\label{lem:finite_termination}
The compression process terminates after finitely many steps, say at $t^\star\in\mathbb N$. Moreover, for every nonempty $T\subseteq V$ and every $|T|\le s\le k$,
\begin{align}\label{eq:termination_property}
    |\{S\in\cA_{t^\star}:T\subseteq S,\ |S|=s\}|\le D^{s-|T|}.
\end{align}
\end{lemma}
\begin{proof}
We track the total number of vertex occurrences in the current family:
\[
    a_i:=\sum_{S\in\cA_i}|S|.
\]
Then $a_i\in\mathbb N$, and
\begin{align*}
    a_i-a_{i+1}
    &=-t_i+\sum_{S\in\cA_i:T_i\subseteq S}|S|\ge -t_i+\sum_{S\in\cA_i:T_i\subseteq S,\ |S|=s_i}|S|>-t_i+s_iD^{s_i-t_i}>0,
\end{align*}
where the last inequality uses $D>1$ and $s_i>t_i$. Thus the process terminates in finitely many steps. At termination, the negation of \Cref{eq:nonsmooth_constraint} gives \Cref{eq:termination_property} for $s>|T|$, while the case $s=|T|$ is immediate.
\end{proof}

The next lemma shows that every earlier set contains some later set. 

\begin{lemma}\label{lem:iterative_avoiding}
For every $i\le j\le t^\star$ and every $S_1\in\cA_i$, there exists $S_2\in\cA_j$ such that $S_2\subseteq S_1$.
\end{lemma}
\begin{proof}
Fix $i$. The statement is trivial for $j=i$. Suppose it holds for $j=m$. For $S_1\in\cA_i$, choose $S_2\in\cA_m$ with $S_2\subseteq S_1$. If $S_2$ is not removed in the $(m+1)$-st compression step, then $S_2\in\cA_{m+1}$. Otherwise $T_m\subseteq S_2\subseteq S_1$, and $T_m\in\cA_{m+1}$. The claim follows by induction.
\end{proof}

Finally, we rule out a one-vertex set in the final family.

\begin{lemma}\label{lem:no_skeleton}
There is no singleton $\{v\}\in\cA_{t^\star}$ when $A\ge 80$.
\end{lemma}
\begin{proof}
Suppose that $\{v\}\in\cA_{t^\star}$. Then for some $i'<t^\star$ and some $s_{i'}>1$,
\[
    |\{S\in\cA_{i'}:v\in S,\ |S|=s_{i'}\}|>D^{s_{i'}-1}.
\]
Define
\[
    b_i:=\sum_{s=2}^k
    |\{S\in\cA_i:v\in S,\ |S|=s\}|D^{k-s}.
\]
By \Cref{eq:nonsmooth_constraint},
\begin{align*}
    b_i-b_{i+1}
    &\ge -\mathbf 1(v\in T_i)D^{k-t_i}
    +\sum_{\substack{S\in\cA_i:\ |S|=s_i,\ T_i\subseteq S}}
      \mathbf 1(v\in S)D^{k-s_i}\ge0.
\end{align*}
Indeed, if $v\notin T_i$, the first term vanishes; if $v\in T_i$, then the second term is strictly larger than $D^{s_i-t_i}D^{k-s_i}=D^{k-t_i}$. Hence $b_i$ is non-increasing. Therefore
\[
    b_0\ge b_{i'}
    >D^{s_{i'}-1}D^{k-s_{i'}}
    =D^{k-1}
    \ge (D/Q\cdot A\Delta^{1/(k-1)})^{k-1}
    \ge\Delta,
\]
where the last inequality holds when $A\ge 80$. On the other hand, since $\cA_0=E$ is $k$-uniform, $b_0=d_H(v)\le\Delta$, a contradiction.
\end{proof}

We now express the final degree hierarchy using the auxiliary hypergraphs from \Cref{sec:overloaded_are_rare}.

\paragraph{Auxiliary hypergraphs.} Fix \(v\) and \(s\). Take every final \(s\)-vertex set containing \(v\), and remove \(v\) from it. These sets form the edge family of \(F^{\mathrm{off}}_{s,v}\), an \((s-1)\)-uniform hypergraph on \(V\setminus\{v\}\):
\begin{align}\label{eq:offline-residual-link}
    E(F^{\mathrm{off}}_{s,v})
    :=\{S\setminus\{v\}:S\in\cA_{t^\star},\ |S|=s,\ v\in S\}.
\end{align}
An auxiliary edge containing \(T\) corresponds to a final set containing \(T\cup\{v\}\). Thus \Cref{lem:finite_termination} gives, for every \(0\le a\le s-1\),
\begin{align}\label{eq:offline-residual-degree}
    \max_{\substack{T\subseteq V\setminus\{v\}\\ |T|=a}}
    d_{F^{\mathrm{off}}_{s,v}}(T)
    \le D^{s-1-a}.
\end{align}
where for $a=0$ the left-hand side equals $e(F^{\mathrm{off}}_{s,v})$.

\subsection{Proofs of the palette-sparsification results}

The degree hierarchy allows us to remove problematic colors while keeping at least half of each list. The LLL then colors every final set properly, and \Cref{lem:iterative_avoiding} applies.

\begin{proof}[Proof of \Cref{thm:palette_sparsification}] 
Let \(\ell=\min\{Q,\lceil K\sqrt{\log n}\rceil\}\).  If \(Q\le\lceil K\sqrt{\log n}\rceil\), then every list is the full palette \([Q]\), and the ordinary LLL coloring bound proves the claim. Hence assume \(\ell=\lceil K\sqrt{\log n}\rceil\).

Choose \(A\) sufficiently large and perform the replacements above. For every final set and each color common to its vertices' lists, record the corresponding bad event:
\[
    \mathcal B
    :=
    \{B_{S,c}: S\in\mathcal A_{t^\star},\ c\in L(S)\}.
\]
These are precisely the events we must avoid to color every final set properly. By \Cref{lem:finite_termination}, \(F_{s,v}^{\mathrm{off}}\) satisfies Assumption~\ref{ass:residual-hierarchy} with \(\Gamma=1\). Here the auxiliary counts are exact, rather than just upper bounds: for every $c\in L(v)$, \(N_s(v,c)=e(F_{s,v,c}^{\mathrm{off}})\). Consequently, \(\widetilde{\mathrm{Bad}}(v)=\mathrm{Bad}(v)\).

We can therefore apply \Cref{lem:palette-trimming}. When \(K\) is sufficiently large, w.h.p. every vertex \(v\) has at least \(\lfloor\ell/2\rfloor\) colors outside \(\mathrm{Bad}(v)\). Keep any \(m=\lfloor\ell/2\rfloor\) such colors in \(G(v)\). Then, for every \(c\in G(v)\) and \(2\le s\le k\),
\[
N_s(v,c)\le \Lambda^{s-1}.
\]
Since \(\Lambda/m\le1/10\) for our choice of constants, \Cref{lem:lll-verification} gives a coloring from the lists \(G(v)\) avoiding all these bad events.

It remains to check the original edges. If an original edge were monochromatic, \Cref{lem:iterative_avoiding} would give a final set contained in it. By \Cref{lem:no_skeleton}, this set has at least two vertices; it would be monochromatic in the same color, contradicting the avoidance of all its bad events. Hence the coloring is proper for \(H\).
\end{proof}

The bounded-rank consequence follows by reducing each edge to one of its $r$-subsets, which brings the problem back to the uniform theorem without increasing the maximum degree.

\begin{proof}[Proof of \Cref{cor:palette_sparsification_bounded}]
    For every edge \(e\in E(H)\), choose any \(r\)-subset \(e'\subseteq e\), and let \(H'\) be the resulting \(r\)-uniform hypergraph after removing repeated edges. A vertex belongs to no more edges in \(H'\) than in \(H\), so \(\Delta(H')\le\Delta(H)\le\Delta\). Apply \Cref{thm:palette_sparsification} to \(H'\) with the same lists. Every edge \(e'\) then contains at least two colors, so the original edge \(e\) containing it does too.
\end{proof}

\section{A deterministic one-pass lower bound}\label{sec:deterministic-lower-bound}

This section proves \Cref{thm:det-streaming-lb} by extending the result of Assadi, Chen, and Sun~\cite{assadi2022deterministic}, originally developed for graphs, to \(k\)-uniform hypergraphs.

Here we sketch the main idea. With \(S\) bits of memory, the algorithm has at most \(2^S\) possible states. Suppose it starts in a fixed state and reads a random subhypergraph of a base hypergraph \(B\), with edges sampled at rate \(\rho\) and maximum degree suitably bounded. Some final state \(\sigma\) must then be reached w.p. at least \(2^{-S}\). Let \(M(\sigma)\) consist of the edges of \(B\) absent from every admissible subhypergraph leading to \(\sigma\). If \(M(\sigma)\) were too large, avoiding all its edges under independent sampling would be too unlikely for \(\sigma\) to occur with probability at least \(2^{-S}\). \Cref{lem:missing-edge-compression} makes this precise and gives \(|M(\sigma)|=O((S+1)/\rho)\).

We iterate this argument over \(p\) stages. Starting from the complete \(k\)-uniform hypergraph, at stage \(i\) we append a low-degree subhypergraph \(H_i\) leading to a chosen state \(\sigma_i\), and use the corresponding missing-edge family \(M_i\) for the next stage after removing a few high-degree vertices. Suppose the final coloring contains a monochromatic \(k\)-set \(e\notin M_p\). At the first stage where \(e\) leaves the missing-edge family, another admissible block containing \(e\) could have produced the same memory state. Replacing that block therefore gives a valid input with the same output coloring but containing the monochromatic edge \(e\), contradicting correctness.

Hence every monochromatic \(k\)-set on the remaining vertices belongs to \(M_p\). Comparing the number of monochromatic \(k\)-sets forced by a coloring with few colors with the upper bound on \(|M_p|\) yields the desired lower bound. We choose \(p=\Theta(\sqrt{\log n/\log\log n})\) while keeping \(\Delta\) polylogarithmic in \(n\). Throughout this section, a \emph{block} is a consecutive portion of the stream containing one subhypergraph, with its edges presented in a fixed order, and an \emph{admissible} subhypergraph is one satisfying the stated maximum-degree restriction.
\begin{lemma}[Missing-edge lemma]\label{lem:missing-edge-compression}
Let $B$ be an $n$-vertex $k$-uniform hypergraph of maximum degree at most $d$, let $0<\rho\le1$, and let \(\mathcal D(B,\rho,d)\) denote the distribution of a random subhypergraph obtained by retaining each edge of \(B\) independently w.p. \(\rho\), conditioned on the resulting hypergraph having maximum degree at most \(2\rho d\). Suppose that $\rho d\ge 4\log(2n)$. Then, for any fixed initial memory state of an $S$-bit deterministic streaming algorithm, there exists a resulting state $\sigma$ such that
\[
|M(\sigma)|\le \frac{\ln 2\,(S+1)}{\rho},
\]
where $M(\sigma)$ is the set of edges of $B$ that occur in no sampled subhypergraph leading to $\sigma$.
\end{lemma}
\begin{proof}
Our proof is built upon Lemma~4.3 of \cite{assadi2022deterministic}, which is stated for graphs, but its proof does not use that an edge has two endpoints. Thus the argument applies without changing the state-counting or missing-edge estimates; we give the details below.

Let \(\widetilde H\) be obtained from \(B\) by retaining each edge independently w.p. \(\rho\), before conditioning. For every vertex \(v\), the random variable \(d_{\widetilde H}(v)\) is binomial with mean \(\rho d_B(v)\le \rho d\). Hence, by a Chernoff bound,
\(
\Prb\left[d_{\widetilde H}(v)>2\rho d\right]\le \exp(-\rho d/3)
\).
Since \(\rho d\ge 4\log(2n)\), a union bound over the \(n\) vertices gives
\[
\Prb\left[\Delta(\widetilde H)\le 2\rho d\right]\ge 1/2.
\]
Now fix the initial memory state. Let \(p_\sigma\) be the probability of reaching state \(\sigma\) under the conditioned distribution \(\mathcal D(B,\rho,d)\). Since there are at most \(2^S\) states, the pigeonhole principle gives one with \(p_\sigma\ge2^{-S}\).

Every subhypergraph leading to this state avoids all edges in \(M(\sigma)\). Before conditioning, independence of edge sampling makes the probability of this avoidance exactly \((1-\rho)^{|M(\sigma)|}\). Conditioning on an event of probability at least \(1/2\) can increase a probability by at most a factor of two. Therefore,
\[
2^{-S}\le p_\sigma\le 2(1-\rho)^{|M(\sigma)|}\le 2\exp\!\left(-\rho|M(\sigma)|\right).
\]
Thus
\(
    |M(\sigma)|\le \frac{\ln 2\,(S+1)}{\rho}
\).
\end{proof}

\begin{proof}[Proof of \Cref{thm:det-streaming-lb}]
Fix a deterministic algorithm using \(O(n\log^a n)\) bits, and let \(S:=\lceil C_0n\log^a n\rceil\) bound its memory for some constant \(C_0>0\). We choose the number of blocks \(p\) and the maximum degree \(\Delta\) as follows:
\[
p:=\left\lfloor\sqrt{\frac{\log n}{\log\log n}}\right\rfloor,
\qquad
\Delta:=\left\lceil C_{k,a}\log^{a+2}n\right\rceil,
\]
where \(C_{k,a}\) is sufficiently large. The \(i\)-th block consists of the edges of a \(k\)-uniform subhypergraph \(H_i\), which we choose after determining the state reached by the preceding blocks. This does not make the final input random or dependent on any algorithmic random choices: the algorithm is deterministic, so the construction produces one fixed worst-case stream.

Let \(V_1=V\), \(B_1=K_n^{(k)}\), and \(d_1=\binom{n-1}{k-1}\). At stage \(i\), the base hypergraph \(B_i\) contains the edges available for constructing the next block, and \(\Delta(B_i)\le d_i\). Suppose the algorithm is currently in state \(\sigma_{i-1}\). Set
\[
\rho_i:=\frac{\Delta}{2pd_i}.
\]
Apply \Cref{lem:missing-edge-compression} to subhypergraphs sampled from $\mathcal D(B_i,\rho_i,d_i)$ starting from $\sigma_{i-1}$. Choose a resulting state $\sigma_i$ whose missing-edge set $M_i$ satisfies
\[
|M_i|\le \frac{\ln 2\,(S+1)}{\rho_i},
\]
and append any admissible subhypergraph $H_i$ leading to $\sigma_i$. Then $\Delta(H_i)\le 2\rho_i d_i=\Delta/p$.

Define
\[
d_{i+1}:=\frac{2k\ln 2\,(S+1)p}{\rho_i n},
\qquad
V_{i+1}:=\{v\in V_i:d_{M_i}(v)\le d_{i+1}\},
\qquad
B_{i+1}:=\{e\in M_i:e\subseteq V_{i+1}\}.
\]
By the definition of $V_{i+1}$, we have $\Delta(B_{i+1})\le d_{i+1}$, so the maximum-degree invariant needed in the next stage is preserved.
Since every removed vertex has $M_i$-degree greater than $d_{i+1}$,
\[
|V_i\setminus V_{i+1}|d_{i+1}\le k|M_i|\le \frac{k\ln 2\,(S+1)}{\rho_i}=\frac{n}{2p}d_{i+1}.
\]
Thus $|V_i\setminus V_{i+1}|\le n/(2p)$ and hence $|V_{p+1}|\ge n/2$. Moreover,
\[
B_{i+1}\subseteq M_i\subseteq B_i,
\]
so every later base hypergraph lies inside every earlier missing-edge set. Moreover, $H_i\cap M_i=\emptyset$ by the definition of $M_i$. Hence, whenever $i<j$, we have $H_j\subseteq B_j\subseteq M_i$ but $H_i\cap M_i=\emptyset$. Thus $H_1,\ldots,H_p$ are pairwise edge-disjoint, and their union has maximum degree at most $\sum_{i=1}^p\Delta(H_i)\le\Delta$.

Let $\chi$ be the coloring output from the final state $\sigma_p$, using $q$ colors. We claim that every monochromatic $k$-set $e\subseteq V_{p+1}$ lies in $M_p$. Otherwise, since $B_1$ is complete and $e\subseteq V_{p+1}$, whenever $e\in M_i$ we also have $e\in B_{i+1}$; hence there is a first $i$ with $e\in B_i\setminus M_i$. By the definition of $M_i$, there exists an admissible subhypergraph $H_i'$ containing $e$ that also takes $\sigma_{i-1}$ to $\sigma_i$. Replacing $H_i$ by $H_i'$ therefore leaves the final memory state, and hence $\chi$, unchanged. The modified stream remains valid. Indeed, for every $j<i$, we have $H_i'\subseteq B_i\subseteq M_j$ and $H_j\cap M_j=\emptyset$, so $H_i'$ is disjoint from all earlier blocks. Also, $H_i'\cap M_i=\emptyset$, whereas every later block lies in $M_i$, so it is disjoint from all later blocks. Finally, $\Delta(H_i')\le\Delta/p$, so the union still has maximum degree at most $\Delta$. Thus the modified input contains the monochromatic edge $e$, a contradiction.

There are at least \(n/2\) surviving vertices. 
If \(q\le n/(4k)\), we have
\begin{align*}
    \sum_{c=1}^q\binom{|\chi^{-1}(c)\cap V_{p+1}|}{k}&\ge q\binom{\lfloor n/(2q)\rfloor}{k}\tag{convexity}\\
    &\ge q\cdot \frac{(n/4q)^k}{k!}=\frac{n^k}{4^kk!q^{k-1}}.\tag{$ n/(2q)\ge 2k$}
\end{align*}
All these monochromatic $k$-sets lie in $M_p$, and therefore
\[
\frac{n^k}{4^kk!q^{k-1}}\le |M_p|\le \frac{\ln 2\,(S+1)}{\rho_p}.
\]
To use this inequality, we estimate how small the last missing-edge family is. Substituting the definition of \(\rho_i\) into the recurrence for \(d_{i+1}\) gives
\[
d_{i+1}=\frac{4k\ln 2\cdot p^2(S+1)}{n\Delta}d_i,\qquad d_1=\binom{n-1}{k-1}\le n^{k-1}
\]
so
\[
d_p\le n^{k-1}\left(\frac{4k\ln 2\cdot p^2(S+1)}{n\Delta}\right)^{p-1},
\qquad
\rho_p\ge \frac{\Delta}{2pn^{k-1}}\left(\frac{n\Delta}{4k\ln 2\cdot p^2(S+1)}\right)^{p-1}.
\]
Substitution yields
\[
q\ge \left(\frac{n\Delta}{C_kp^2(S+1)}\right)^{p/(k-1)},
\]
where $C_k$ is a constant depending on $k$.
Since $S=O(n\log^a n)$ and $\Delta=\Theta(\log^{a+2}n)$,
\[
\frac{n\Delta}{C_kp^2(S+1)}=\Omega_{k,a,C_0}(\log n\log\log n),
\]
and consequently
\[
\log q=\Omega_{k,a}\!\left(\sqrt{\log n\log\log n}\right).
\]
Taking, for example, any fixed $c<1/(2(a+2))$, we have $\Delta^c=o(\sqrt{\log n\log\log n})$, and hence $q\ge \exp(\Delta^c)$ for all sufficiently large $n$. 

If instead \(q>n/(4k)\), the desired bound already follows because \(\exp(\Delta^c)=n^{o(1)}\).

Finally, the repeated applications of \Cref{lem:missing-edge-compression} require \(\rho_i d_i\ge4\log(2n)\) and \(\rho_i\le1\). The first holds because \(\rho_i d_i=\Delta/(2p)\gg\log n\). For the second, observe that
\[
\frac{d_{i+1}}{d_i}=\Theta_{k,a,C_0}\!\left(\frac{1}{\log n\log\log n}\right).
\]
Thus $d_i=n^{k-1-o(1)}\gg\Delta/p$ for all $i\le p$, so $\rho_i\le1$. This completes the proof.
\end{proof}




\section*{AI disclosure}
We used ChatGPT to assist with paper writing, to quickly test intuitions, and to check for typos. It found that the lower bound for palette sparsification was already established in prior literature. It also helped establish the deterministic one-pass lower bound. The authors take full responsibility for all statements and proofs.

\newpage
\phantomsection\addcontentsline{toc}{section}{References}
\bibliographystyle{alpha}
\bibliography{References}

\newpage
\appendix

\section{Proof of concentration}
\label{sec:proof-of-concentration}

\begin{proof}[Proof of \Cref{lem:literal-loads}]
Fix \(v,c,s\) and condition on \(c\in L(v)\).  For every \(u\neq v\), let
\[
    Y_{u,c}:=\mathbf 1[c\in L(u)].
\]
Then the variables \(Y_{u,c}\) are mutually independent over \(u\neq v\), with
\[
    \Prb[Y_{u,c}=1]=\frac{\ell}{Q}.
\]
Write the edge count as
\[
    X_s(v,c):=e(F_{s,v,c})
    =\sum_{B\in E(F_{s,v})}\prod_{u\in B}Y_{u,c}.
\]
For each edge \(B\), the product is one exactly when all its vertices have \(c\) available. Assumption~\ref{ass:residual-hierarchy} gives
\[
    \max_{|T|=a}d_{F_{s,v}}(T)\le \Gamma D^{s-1-a},
    \qquad 1\le a\le s-1,
\]
and 
\[
e(F_{s,v})\le D^{s-1}.
\]  
Hence
\[
    \E\left[\Lambda^{1-s}X_s(v,c)\right]
    \le
    \Lambda^{1-s}D^{s-1}\left(\frac{\ell}{Q}\right)^{s-1}
    =
    \left(\frac{D\ell}{\Lambda Q}\right)^{s-1}.
\]

We next bound the \(h\)-th moment for an integer \(h\ge2\). Expanding \(X_s(v,c)^h\) gives a sum over ordered \(h\)-tuples of auxiliary edges, with repetitions allowed. Since each indicator satisfies \(Y_{u,c}^j=Y_{u,c}\) for \(j\ge1\), the expectation of a term depends only on the union of its edges. Fix the first \(h-1\) edges
\[
    B_1,\ldots,B_{h-1}\in E(F_{s,v}),
\]
and put
\[
    U:=\bigcup_{i=1}^{h-1}B_i,
    \qquad |U|\le(h-1)(s-1).
\]
If the final edge is \(B\), the only new vertices are those in \(B\setminus U\). Its extra factor in the expectation is therefore \((\ell/Q)^{(s-1)-|B\cap U|}\). We sum this factor by considering the size \(a\) of the intersection with \(U\). For each \(a\ge1\), there are \(\binom{|U|}{a}\) possible intersection sets, and each is contained in at most \(\Gamma D^{s-1-a}\) auxiliary edges. Counting all extensions of each such set gives an upper bound:
\begin{align*}
    \sum_{B\in E(F_{s,v})}\left(\frac{\ell}{Q}\right)^{(s-1)-|B\cap U|}&\le D^{s-1}\left(\frac{\ell}{Q}\right)^{s-1}+\Gamma\sum_{a=1}^{s-1}
    \binom{|U|}{a}D^{s-1-a}\left(\frac{\ell}{Q}\right)^{s-1-a}\\
    &=\left(\frac{D\ell}{Q}\right)^{s-1}
    \left[1+\Gamma\sum_{a=1}^{s-1}\binom{|U|}{a}\left(\frac{Q}{D\ell}\right)^a\right]\\
    &\le\left(\frac{D\ell}{Q}\right)^{s-1}\left[1+\Gamma\left(\left(1+\frac{Q}{D\ell}\right)^{|U|}-1\right)\right]\\
    &\le\left(\frac{D\ell}{Q}\right)^{s-1}\left[1+\Gamma\left(\left(1+\frac{Q}{D\ell}\right)^{(h-1)(s-1)}-1\right)\right].
\end{align*}
The bound holds regardless of the first \(h-1\) edges. Applying it successively to the expansion gives
\[
\E\left[\left(\Lambda^{1-s}X_s(v,c)\right)^h\right]
\le
\left(\frac{D\ell}{\Lambda Q}\right)^{h(s-1)}
\left[
1+\Gamma\left(
\left(1+\frac{Q}{D\ell}\right)^{(h-1)(s-1)}-1
\right)
\right]^h.
\]
Take
\[
h:=\left\lfloor\frac{\ell}{320\Gamma}\right\rfloor,
\]
which is at least \(2\) for sufficiently large \(\ell\). Since
\[
\frac{D\ell}{\Lambda Q}=\frac12
\qquad\text{and}\qquad
\frac{Q}{D\ell}=\frac{80}{\ell},
\]
we have
\[
\left(1+\frac{Q}{D\ell}\right)^{(h-1)(s-1)}
\le
\exp\left(\frac{80h(s-1)}{\ell}\right)
\le
e^{(s-1)/(4\Gamma)}.
\]
Moreover, since $\Gamma\ge 1$, by Bernoulli's inequality we have \(1+\Gamma t\le (1+t)^{\Gamma}\) for \(t\ge0\). Hence
\[
1+\Gamma\left(e^{(s-1)/(4\Gamma)}-1\right)\le e^{(s-1)/4}.
\]
Therefore
\[
\E\left[\left(\Lambda^{1-s}X_s(v,c)\right)^h\right]
\le
\left(\frac{e^{1/4}}{2}\right)^{h(s-1)}.
\]
The base \(\beta:=e^{1/4}/2\) is less than one, so this moment decreases exponentially with \(h(s-1)\). Markov's inequality gives
\[
\Prb\left[X_s(v,c)>\Lambda^{s-1}\,\middle|\,c\in L(v)\right]
\le
\beta^{h(s-1)}.
\]
A union bound then gives
\[
\Prb[c\in\widetilde{\mathrm{Bad}}(v)\mid c\in L(v)]
\le
\sum_{s=2}^k\beta^{h(s-1)}
\le
\frac{\beta^h}{1-\beta^h}
=
e^{-\Omega_{\Gamma}(\ell)}.
\qedhere
\]
\end{proof}

\section{Proof of streaming coloring for bounded-rank hypergraphs}\label{sec:bounded_rank_proof}

We group input edges by size and apply the uniform construction to each group, using the same color lists. We then simplify all the resulting constraints together. There are only a fixed number of edge sizes, so combining the space bounds and the auxiliary codegree bounds changes only constants depending on \(r\) and \(k\).

\begin{proof}[Proof of \Cref{cor:bounded-rank}]
Fix sufficiently large constants \(A_{r,k},K_{r,k},R_{r,k}\) depending only on \(r\) and \(k\), and set
\[
Q:=\left\lceil A_{r,k}\Delta^{1/(r-1)}\right\rceil,
\qquad
\ell:=\left\lceil K_{r,k}\sqrt{\log n}\right\rceil.
\]

For \(r\le q\le k\), let \(E_q:=\{e\in E:|e|=q\}\) and let \(d_q(S):=|\{e\in E_q:S\subseteq e\}|\) count the \(q\)-vertex edges containing \(S\). We use the storage and sampling procedure from \Cref{thm:uniform} for each group \(E_q\), with the global palette and lists chosen above. When a \(q\)-vertex edge contains an active \(s\)-vertex set \(S\), \(2\le s<q\), we count this occurrence w.p.
\[
\rho_{q,s}:=\frac{R_{r,k}\log n}{Q^{q-s}},
\]
maintaining separate counters \(C_{q,S}\) for different input edge sizes \(q\). Let \(\widehat{\calP}^{(q)}_s\) be the resulting candidate family. After the stream, combine the candidate constraints from all groups with the original full-edge constraints and process them in increasing order of set size. 
Thus a constraint on a smaller set can remove a larger constraint.


Correctness follows as before: each original edge contributes its full-edge constraints, and a constraint is discarded only when a smaller one for the same color has already been kept.

Keeping separate counters for each edge size changes the total by only a factor depending on \(r,k\).  For every \(r\le q\le k\), the expected number of full-edge constraints is
\[
O_{r,k}\left(n\Delta\frac{\ell^q}{Q^{q-1}}\right)=\widetilde O_{r,k}(n),
\]
and, for every \(2\le s<q\), the expected number of sampled active \((q,s)\)-core occurrences is
\[
\binom{q}{s}|E_q|
\cdot
\frac{\ell^s}{Q^{s-1}}
\cdot
\frac{R_{r,k}\log n}{Q^{q-s}}
\le
O_{r,k}\!\left(
n\Delta\frac{\ell^s\log n}{Q^{q-1}}
\right)
=
\widetilde O_{r,k}(n),
\]
since \(q\ge r\) and \(Q^{r-1}=\Theta_{r,k}(\Delta)\). The remaining space argument is identical to \Cref{lem:space}. 

Next we combine the auxiliary hypergraphs for the different input edge sizes. Their union will bound all the constraints involving a given vertex and color. For fixed \(v,s\) and \(q\ge\max\{r,s\}\), define \(F^{(q)}_{s,v}\) as in \Cref{sec:loads}, replacing \(d_H\), \(k\), and \(Q^{k-s}\) by \(d_q\), \(q\), and \(Q^{q-s}\), respectively, and let
\[
F_{s,v}:=\bigcup_{q=\max\{r,s\}}^k F^{(q)}_{s,v}.
\]
The double-counting argument of \Cref{lem:streaming-residual-hierarchy} gives, for \(1\le a\le s-2\),
\[
e(F_{s,v})\le C_{r,k}\frac{\Delta}{Q^{r-1}}Q^{s-1},
\qquad
\max_{|T|=a}d_{F_{s,v}}(T)\le C_{r,k}Q^{s-a-1},
\]
for some constant \(C_{r,k}\). Thus, by choosing the constant in \(Q=\Theta_{r,k}(\Delta^{1/(r-1)})\) sufficiently large,
\[
e(F_{s,v})\le D^{s-1},
\qquad
\max_{|T|=a}d_{F_{s,v}}(T)\le \Gamma_{r,k}D^{s-1-a}
\]
for some constant \(\Gamma_{r,k}\); the case \(a=s-1\) is trivial. Hence Assumption~\ref{ass:residual-hierarchy} holds. Every kept \(s\)-core constraint \(B_{S,c}\) arises from some \(q\) with \(S\in\widehat{\calP}^{(q)}_s\). Applying \Cref{lem:active-core-detection} separately to every \((q,s)\) and taking a union bound gives a joint detection event with the same high-probability guarantee given any sampled list assignment $L$. On the joint detection event, the same argument as in \Cref{lem:domination} gives
\[
S\setminus\{v\}\in E(F^{(q)}_{s,v,c})\subseteq E(F_{s,v,c}).
\]
The same argument applies to kept full-edge constraints with \(q=s\). Hence
\[
N_s(v,c)\le e(F_{s,v,c}),
\]
on the joint detection event.

Thus the auxiliary edge counts bound the local constraint counts, just as in the uniform case. We can now remove problematic colors and apply the final coloring step. By \Cref{lem:palette-trimming}, w.h.p. every vertex has at least \(\lfloor\ell/2\rfloor\) colors outside \(\widetilde{\mathrm{Bad}}(v)\). When detection succeeds for all edge sizes, the comparison above gives \(\mathrm{Bad}(v)\subseteq\widetilde{\mathrm{Bad}}(v)\). We therefore arbitrarily choose a sublist \(G(v)\subseteq L(v)\setminus \mathrm{Bad}(v)\) of size \(\lfloor\ell/2\rfloor\). Then \Cref{lem:lll-verification} gives an assignment avoiding all kept bad events, and \Cref{lem:kernel-correct} shows that it properly colors the original hypergraph.

As in the proof of \Cref{thm:uniform}, we cap the memory of each copy at a constant multiple of its expected-space bound, run \(O(\log n)\) independent copies, and process the surviving copies sequentially. This gives the claimed \(\widetilde O_{r,k}(n)\) space bound and success probability, with expected polynomial postprocessing time.

Finally, if the palette is too small or one of the displayed sampling probabilities exceeds one, the maximum degree is bounded by a fixed power of \(\log n\). As in the proof of \Cref{thm:uniform}, we then store the entire hypergraph and color it directly. 
\end{proof}

\end{document}